\documentclass[12pt]{amsart}
\usepackage{biblatex}
\usepackage{mathrsfs}
\usepackage{amsmath,amsthm}
\usepackage{amssymb}
\usepackage{pinlabel}
\usepackage{graphicx}
\usepackage[arrow,matrix,curve]{xy}\SilentMatrices
\usepackage{amscd}
\usepackage{times}
\usepackage{enumerate}
\usepackage{indentfirst}
\usepackage{stmaryrd}
\usepackage{xypic}
\usepackage{mathtools}
\usepackage{euscript}
\usepackage{extarrows}

\usepackage{xcolor}

\newtheorem{thm}{Theorem}[section]
\newtheorem*{thm*}{Theorem}
\newtheorem{lem}[thm]{Lemma}
\newtheorem{cor}[thm]{Corollary}

\newtheorem{prop}[thm]{Proposition}
\theoremstyle{definition}
\newtheorem{defin}[thm]{Definition}
\newtheorem*{rmk}{Remark}

\newcommand{\Z}{\mathbb Z}
\newcommand{\F}{\EuScript F}
\renewcommand{\hom}{\mathrm{Hom}}

\renewcommand{\L}{\EuScript L}
\newcommand{\f}{\mathbf{f}}
\renewcommand{\r}{\mathbf{r}}
\renewcommand{\c}{\mathbf{c}}
\renewcommand{\a}{\alpha}
\newcommand{\g}{\mathbf{g}}

\newcommand{\mono}{\: \ar@{^{(}->}}
\newcommand{\epi}{\ar@{->>}}

\renewcommand{\lim}{\mathrm{lim}}

\newcommand{\defeq}{\coloneqq}
\newcommand{\xym}{\xymatrix}
\newcommand{\ra}{\, \Rightarrow\, }

\newcommand{\ab}{\mathrm{ab}}
\newcommand{\e}{\varepsilon}
\newcommand{\w}{\wedge}

\newcommand{\D}{\Delta}

\renewcommand{\k}{\mathbf{k}}

\newcommand{\W}{\overline{W}}
\newcommand{\CC}{\EuScript C}

\newcommand{\OO}{\EuScript O}

\newcommand{\FF}{\mathcal F}

\newcommand{\ZZ}{\EuScript Z}
\newcommand{\BB}{\EuScript B}
\newcommand{\GG}{\mathcal G}

\renewcommand{\tilde}{\widetilde}

\newcommand{\mr}{\mathrm}

\renewcommand{\ab}{\mathrm{Ab}}
\def\BB{\text{\,\sf\reflectbox{B}}}

\bibliography{references}

\def\xyma{\xymatrix@M.7em}
\begin{document}
\title{Limits, standard complexes and ${\bf fr}$-codes}
\author{Sergei O. Ivanov}
\address{Laboratory of Modern Algebra and Applications, St. Petersburg State University, 14th Line, 29b,
Saint Petersburg, 199178 Russia}\email{ivanov.s.o.1986@gmail.com}
\author{Roman Mikhailov}
\address{Laboratory of Modern Algebra and Applications, St. Petersburg State University, 14th Line, 29b,
Saint Petersburg, 199178 Russia and St. Petersburg Department of
Steklov Mathematical Institute} \email{rmikhailov@mail.ru}
\author{Fedor Pavutnitskiy}
\address{Laboratory of Modern Algebra and Applications, St. Petersburg State University, 14th Line, 29b,
Saint Petersburg, 199178 Russia} \email{fedor.pavutnitskiy@gmail.com}
\begin{abstract}
For a strongly connected category $\CC$ with pair-wise coproducts, we introduce a cosimplicial object, which serves as a sort of resolution for computing higher derived functors of
${\sf lim} : \ab^{\CC}\to \ab$. Applications involve K\"unneth theorem for higher limits and ${\sf lim}$-finiteness of $\f\r$-codes. A dictionary for the $\f\r$-codes with words of length $\leq 3$ is given.
\end{abstract}
\thanks{
The authors are supported by the Russian Science Foundation grant N 16-11-10073.}
\maketitle
\section{Introduction}
Let $G$ be a group. By $\mr{Pres}(G)$ we denote the \textit{category of presentations} of $G$ with objects being free groups $F$ together with epimorphisms to $G$.
Morphisms are group homomorphisms over $G$. For a functor $\mathcal F:\mr{Pres}(G)\to \mr{Ab}$ from the category $\mr{Pres}(G)$ to the category of abelian groups, one can
consider the (higher) limits ${\sf lim}^i\, \mathcal F,\ i\geq 0,$ over the category of presentations.
The limits ${\sf lim}^i\, \mathcal F$ are
studied in the series of papers \cite{ivanov2015higher}, \cite{ivanov2017fr}, \cite{mikhailov2019dimension}, \cite{mikhailov2016generalized}, \cite{mikhailov2018free}.

Let $\mr{Ring}$ be the category of rings. The group ring functor $\Z[-]:\mr{Pres}(G)\to\mr{Ring}, \ (F\twoheadrightarrow G)\mapsto \Z[F]$ has two functorial ideals $\f$ and $\r$ defined as
$$
\f (F\twoheadrightarrow G)= {\sf ker}\{\Z[F]\to\Z\}, \ \r (F\twoheadrightarrow G)= {\sf ker}\{\Z[F]\to\Z[G]\}
$$
For different products of ideals $\f$ and $\r$, their sums and intersections, like
\begin{equation}\label{fr-code_example}
\f\r+\r\f, \ \r^2\cap \f^3
\end{equation}
one can consider their higher limits. It turns out, such limits, which depend functorially on $G$, cover a rich collection of various functors on the category of groups, including certain homological functors,
derived functors etc.

(Finite) sums of monomials formed from letters $\f$ and $\r$ we call ${\bf fr}${\it-sentences} or ${\bf fr}${-codes}.
By {\it translation} we mean a description of the functors ${\sf lim}^i\left({\bf fr}-\mathrm{code}\right)$, $i\geq 1$, ${\bf fr-}\mathrm{codes}$ viewed as
functors $\mr{Pres}(G)\to \mr{Ab}$. For the moment we do not have a unified method of translation of a given code and, in every new case, in order to translate a code,
we find specific tricks. At the end of the paper we present a {\it dictionary} of all nontrivial translations of codes with monomials of length $\leq 3$. In order to illustrate the
diversity of functors which appear in this way, we give the following examples:
\begin{align*}
& {\sf lim}^1({\bf rff+frr})={\sf Tor}(H_2(G),G_{ab}),\\
& {\sf lim}^1({\bf rr+frf+rff})=H_2(G,G_{ab}),\\
& {\sf lim}^1({\bf rr+frf})=H_3(G),\\
& {\sf lim}^2({\bf rr+frf})={\bf g}\otimes_{\mathbb Z[G]}{\bf g}.
\end{align*}
Here $H_i(G)$ is the $i$th integral homology of $G$, ${\bf g}$ the augmentation ideal in $\mathbb Z[G]$, $G_{ab}$ the abelianization of $G$.

Since the category $\mr{Pres}(G)$ is strongly connected, the ${\sf lim}^1({\bf fr}\mr{-code})$ has a natural interpretation as the maximal constant subfunctor of $\f/({\bf fr}\mr{-code})$ (see \cite{ivanov2015higher}, \cite{ivanov2017fr}). For example,
\begin{align*}
& {\sf lim}^1{\bf rr+fff}={\sf lim}^0\frac{\bf f}{\bf rr+fff}=\frac{({\bf rf+fff})\cap ({\bf fr+fff})}{\bf rr+fff}={\sf Tor}(G_{ab}, G_{ab}),\\ 
& {\sf lim}^1{\bf rf+fr}={\sf lim}^0\frac{\bf f}{\bf rf+fr}=\frac{\bf ff}{\bf rf+fr}={\bf g}\otimes_{\mathbb Z[G]}{\bf g}. 
\end{align*}

The point of this theory (which we also call metaphorically as ${\bf fr}${-\it language}), is that the formal manipulations with codes in two letters may induce deep and
unexpected transformations of functors. Simple transformations of ${\bf fr}$-codes, like changing the symbol $\bf r$ by $\bf f$ in a certain place, adding a monomial to the ${\bf fr}$-code etc,
induce natural transformations
of (higher) limits determined by these ${\bf fr}$-codes. For example, the transformation of the ${\bf fr}$-codes
$$
{\bf rr+frf}\rightsquigarrow {\bf rr+frf+rff}
$$
induces the natural transformation of functors
$$
H_3(G)={\sf lim}^1({\bf rr+frf})\rightsquigarrow H_2(G,G_{ab})={\sf lim}^1({\bf rr+frf+frr}).
$$
Here the map $H_3(G)\to H_2(G, G_{ab})$ is constructed as $H_3(G)=H_2(G,{\bf g})\to H_2(G,{\bf g}/{\bf g}^2)=H_2(G,G_{ab}),$ where the last map is induced by the natural projection
${\bf g}\twoheadrightarrow {\bf g}/{\bf g}^2=G_{ab}.$

This paper has two main parts. The first part is more abstract, we prove that any (finite) $\f\r$-code has only finite number of non-zero higher limits
(see Theorem \ref{fr-codes_are_finite}). In order to prove this statement, we develop a general theory of standard complexes constructed for elements of categories with pairwise
coproducts (such as our category $\mr{Pres}(G)$). More precisely, for any object $c$ of a category with pairwise coproducts we introduce a cosimplicial object $\BB(c),$ such that,
for any functor $\mathcal F$ from our category to abelian groups, the (higher) limits ${\sf lim}^i \mathcal F$ are naturally isomorphic to the cohomotopy groups $\pi^i\mathcal F\BB(c)$
(Theorem \ref{higher_limits_as_cohomology}). It follows from Theorem \ref{fr-codes_are_finite} that, given a ${\bf fr}$-code, the number of its non-zero
higher limits is finite. In the second part we present concrete translations.
We form a dictionary of the various $\f\r$-codes using spectral sequences,
Gr\"unberg resolution, K\"unneth-type formulas and collections of tricks. Observe that, not all $\f\r$-codes can be translated using homological algebra only, in some cases (like the
case $\bf rr+ffr+frf+rff$),
nontrivial statements from the theory of groups and group rings are useful.

\section{The standard complex}
\begin{defin}\label{strongly_connected_pairwise_coproducts}
 A category $\CC$ is called \textit{strongly-connected} if for any two objects $c,c'\in \CC$
$$
\hom_{\CC}(c,c')\neq \emptyset .
$$
Moreover, if for any $c,c'\in \CC$ there exists coproduct $c\sqcup c'$, we say that $\CC$ is a \textit{category with pairwise coproducts} (i.e. with finite non-empty coproducts).
\end{defin}

\begin{defin}\label{higher_limits}
Let $\mathcal F: \CC\to \mr{Ab}$ be a functor. The \textit{Higher limits} ${\sf lim}^i \,\mathcal F$ of $\mathcal F$ are the right derived functors of the limit functor :
$$
{\sf lim}^i\, \mathcal F =\mathbb R^i {\sf lim} \, \mathcal F, \hspace{1cm} {\sf lim}: \mr{Ab}^{\CC}\to \mr{Ab}.
$$

\end{defin}
We will assume that in the functor category there are enough injective objects, so higher limits of any functor exists, provided $\CC$ is small. In a general case, as in section \eqref{fr-codes_section}, where $\CC=\mr{Pres}(G)$, the existence of higher limits for functors of interest can be established, using Grothendieck-Tarsky theory, as in \cite{ivanov2017fr}.

For a cochain complex of functors the relation between higher limits of its terms and limits of its cohomology is given by the following spectral sequence.
\begin{prop}[\cite{ivanov2015higher}, (2.5), (2.6)]\label{lim_spectral_sequence}
Let $\mathcal F^{\bullet}$ be a bounded below cochain complex of functors with ${\sf lim}$-acyclic cohomology. Then there exists a convergent spectral sequence
\begin{equation}
E_1^{p,q} = {\sf lim}^q\, \mathcal F^p\ra {\sf lim} \, H^{p+q}(\mathcal F^{\bullet})
\end{equation}
with the differential on the first page induced by the differential of $\mathcal F^{\bullet}$.
\end{prop}

\begin{rmk}
For a functor $\mathcal F$ consider its subfunctor of the \textit{invariants} ${\sf inv}\, \mathcal F: \CC\to \mr{Ab}$:
\begin{equation}\label{defin_of_invariants}
{\sf inv}\,\mathcal F(c)=\{x\in \mathcal F(c)|\forall\,c'\in\CC, \ \varphi, \psi:c\to c',\ \mathcal F(\varphi)(x)=\mathcal F(\psi)(x)\}.
\end{equation}
In strongly connected categories this functor is constant and its value is equal to ${\sf lim}\,\mathcal F$, see (4.1) in \cite{ivanov2017fr}.
Moreover, it is known \cite{mikhailov2019dimension} that the limit of a functor from a strongly connected category with pair-wise coproducts is equal to the equalizer
$${\sf lim}\, \mathcal F \cong {\rm eq}(\mathcal F(c)\rightrightarrows \mathcal F(c\sqcup c))) $$
for any $c\in \CC$.
 In particular, this equalizer does not depend on $c$.
\end{rmk}

To generalize the relation between limits and invariants to the level of derived functors we introduce the following notion:
\begin{defin}\label{defin_of_standard_complex}
For $c\in\CC$ consider the following cosimplicial object $\BB:\Delta\to \CC$, which we will call the \textit{standard complex} associated with $c$:
$$
\BB (c)^n=\bigsqcup_{j=0}^{n} c, $$
$$  \BB (c) ([n]\xrightarrow f [m])= \bigsqcup_{j=0}^{n} c \xrightarrow{(i_{f(0)},\dots ,i_{f(n)})} \bigsqcup_{k=0}^{m} c
$$
here $i_j: c\to \bigsqcup_{k=0}^{m} c, \, 0\leq j\leq m$ are canonical inclusions and notation $(g_0,\dots ,g_n), \ g_k: c\to c'$ stands for the unique map $c^{\sqcup\, n+1} \to c'$ such that $(g_0,\dots ,g_n)\circ i_j = g_j$.
\end{defin}
By definition, cofaces and codegeneracies of $ \BB(c)$
$$d^j:c^{\sqcup\, n+1} \to c^{\sqcup\, n+2}, \hspace*{1cm} s^j: c^{\sqcup\, n+1} \to c^{\sqcup\, n} $$
 are given by
\begin{equation}\label{cofaces_and_codegenracies_in_the_standard_complex}
d^j = (i_0, \dots , \hat i_j, \dots ,i_{n+1}), \ s^j = (i_{0} ,\dots , i_{j},i_{j} ,\dots ,i_{n}).
\end{equation}
This complex is very similar to the so-called \textit{canonical resolution}, associated with the monad  $(c\sqcup(-),\nabla, i_2)$, here $\nabla: c\sqcup c\sqcup (-)\xrightarrow{(i_1,i_2)\sqcup \mr{id}} c\sqcup (-)$, see \cite{weibel1995introduction}  (8.6.8). This similarity will become an identification, if there is an initial object $0$ in $\CC$. In this case though all higher limits of the functor$\mathcal F:\CC\to \mr{Ab}$ are trivial, provided $\mathcal F(0)=0$. Alternatively, since $(\CC,\sqcup)$ can be considered as a strong monoidal category (without unit) and every object is a monoid with respect to this structure, for any $c$ the standard resolution $\BB(c)$ can be considered as a unique monoidal functor $\Delta\to \CC$ which sends $[0]$ to $c$, as in (7.5) of \cite{mac2013categories}.

Now we will study some homotopical properties of the standard complex $\BB(c)$.
\begin{defin}[\cite{meyer1990cosimplicial}, (2.1)]
Let $f,g: X\to Y$ be two morphisms between cosimplicial objects $X$ and $Y$. A \textit{cosimplicial homotopy} between $f$ and $g$ is a collection of maps $k^i: X^{n+1}\to Y^n, \ 0\leq i \leq n$, satisfying the following identities:
\begin{gather}
k^0 d^0 = g, \ k^n d^{n+1} = f\\
k^jd^i =
\begin{cases}
d^ik^{j-1}, & \ i<j \\  \label{homotopy_coface_identity1}
k^{j-1}d^j, & \ i=j>0 \\
d^{i-1}k^j, & \ i>j+1
\end{cases}\\
k^j s^i =
\begin{cases}
s^ik^{j+1}, & \ i\leq j \\
s^{i-1}k^j, & i>j
\end{cases}
\end{gather}
\end{defin}
We will use the following definition of the Moore complex and the alternate sum complex for the abelian case, which are dualizations of the standard definitions, as in \cite{goerss-jardine}:
\begin{defin}
Let $A$ be a cosimplicial object in an abelian category $\CC$
\begin{itemize}
\item The \textit{Moore complex} $QA$ of $A$ is a cochain complex
$$
(QA)^n={\sf coker}\{\bigoplus_{i=1}^n A^{n-1}\xrightarrow{d^i} A^n\}
$$

\item The \textit{alternate sum complex} $CA$ of $A$ is a cochain complex
$$
CA^n = A^n, \ d = \sum_{i=0}^{n+1}(-1)^i d^i
$$
\end{itemize}
\end{defin}
Both constructions are functorial, with $Q: \CC^{\Delta}\to \mr{Ch}_{\leq 0}(\CC)$ being an exact functor, and as in the simplicial case, these two complexes are chain homotopic to each other. Since a cosimplicial homotopy $\{k^i\}_{i=0}^{\infty}$ between $f$ and $g$ induces a chain homotopy
$$
k = \sum_{i=0}^n (-1)^i k^i
$$
between $Cf$ and $Cg$, $Qf$ and $Qg$ are also homotopic.

The Moore complex $QA$ also has a convenient iterative description in terms of the  \textit{d'ecalage} of $A$, which is a cosimplcial object $\mr{Dec}\, A$ with the following structure:
$$
(\mr{Dec}\, A)^n = A^{n+1}, \ d^i_{ \mr{Dec}\, A} = d^{i+1}_A, \ s^j_{\mr{Dec}\, A} = s^{j+1}_A .
$$
\begin{prop}\label{iterative_description_of_moore_complex}
The following formula holds:
$$
(QA)^n = {\sf coker}\, \{ (QA)^{n-1}\xrightarrow{d^1} (Q\mr{Dec}\,A)^{n-1}\}.
$$
\end{prop}
\begin{proof}
Consider the following diagram:
$$
\xym{
	\bigoplus_{i=1}^{n-1}A^{n-2}\ar[r]^{d^i}\ar[d]^{d^1}&A^{n-1}\ar[d]^{d^1}\epi[r]&(QA)^{n-1}\ar[d]^{d^1}\\
	\bigoplus_{i=2}^{n-1}A^{n-1}\ar[r]^{d^i}&A^n\epi[rd]\epi[r]&(Q\mr{Dec}\, A)^{n-1}\epi[d]\\
	&&(QA)^n
	}
$$
The diagonal arrow here represents a map to a ``total'' cokernel of the square (the cokernel of the natural map from the push-out to the right-bottom corner), which is equal to a ``sequential'' cokernel, represented by the rightmost vertical arrows.
\end{proof}
Turns out, on a strongly connected $\CC$ the standard complex construction is constant up to homotopy:
\begin{thm}\label{hatC_homotopy_constant}
Let $\CC$ be a category with pair-wise coproducts. Then for any two maps $f,g : c\to c'$  the induced morphisms $\BB(c)\to \BB(c')$ are homotopic.
\end{thm}
\begin{proof}
Consider the following collection of maps $\{k^i:\BB(c)^{n+1}\to \BB(c')^{n}\}_{i=0}^{ \infty}$:
\begin{equation}\label{fg_homotopy}
k^i = (i_0 f,\dots , i_{i} f, i_{i} g , \dots , i_{n} g) = s^i \alpha^i, \ \alpha^i\defeq \underbrace{f\sqcup\dots\sqcup f}_{i+1}\sqcup\underbrace{g\sqcup\dots\sqcup g}_{n+1-i}
\end{equation}
First we consider how $\alpha^j$ commutes with cofaces and codegeneracies. For fixed $i<j$:
\begin{gather*}
\alpha^jd^i = (i_0 f,\dots , i_{j} f, i_{j+1} g,\dots ,i_{n+1} g)(i_0,\dots ,\hat i_{i},\dots ,i_{n+1}) = (i_0 f,\dots ,\widehat{ i_{i} f},\dots, i_{j} f, i_{j+1} g,\dots ,i_{n+1} g)\\=
(i_0,\dots ,\hat i_{i},\dots ,i_{n+1})(i_0 f,\dots ,i_{j-1} f, i_{j} g,\dots ,i_{n} g) = d^i\alpha^{j-1}
\end{gather*}
For $i>j+1$:
$$
\alpha^jd^i =  (i_0 f,\dots ,i_{j} f, i_{j+1} g,\dots ,\widehat{i_{i} g},\dots ,i_{n+1} g)=d^i\alpha^j
$$
For codegeneracies if $i\leq j$:
\begin{gather*}
\alpha^js^i = (i_0 f,\dots ,i_{j} f, i_{j+1} g,\dots ,i_{n+1} g)(i_0,\dots i_{i}, i_{i},\dots ,i_{n})\\=
(i_0 f,\dots ,i_{i}f, i_{i} f,\dots ,i_{j} f, i_{j+1} g,\dots ,i_{n+1} g)=s^i\alpha^{j+1}
\end{gather*}
And similarly for $i>j$:
$$
\alpha^js^i=(i_0 f,\dots ,i_{j} f, i_{j+1} g,\dots ,i_{i} g, i_{i}g,\dots ,i_{n+1} g) = s^i \alpha^j
$$
Returning to $k^i$ and using the cosimplicial identities :
\begin{gather*}
k^jd^i = s^j\alpha^j d^i =
\begin{cases}
s^jd^i\alpha^{j-1}, & i<j\\
s^jd^i\alpha^j, & i>j+1
\end{cases}\\=
\begin{cases}
d^is^{j-1}\alpha^{j-1}, & i<j\\
d^is^j\alpha^j, & i>j+1
\end{cases}
=
\begin{cases}
d^ik^{j-1}, & i<j\\
d^ik^j, & i>j+1
\end{cases}\\
k^js^i= s^j\alpha^j s^i =
\begin{cases}
s^j s^i  \alpha^{j+1}, & i\leq j\\
s^ j s^i \alpha^i, & i<j
\end{cases}\\=
\begin{cases}
s^i s^{j+1}  \alpha^{j+1}, & i\leq j\\
s^{i-1} s^j \alpha^j, & i<j
\end{cases}=
\begin{cases}
s^i k^{j+1}, & i\leq j\\
s^{i-1} k^j, & i<j
\end{cases}
\end{gather*}
Finally we consider relations for $k^jd^j$ and the boundaries of the homotopy $k^0d^0, \ k^nd^{n+1}$:
\begin{gather*}
k^jd^j = (i_0 f,\dots ,i_{j} f, i_{j} g,\dots ,i_{n}g)(i_0,\dots ,\hat i_{j},\dots ,i_{n+1}) =
(i_0 f,\dots ,i_{j-1} f, i_{j} g,\dots ,i_{n}g)\\
=  (i_0 f,\dots ,i_{j-1} f, i_{j-1} g,\dots ,i_{n}g)(i_0,\dots ,\hat i_{j},\dots ,i_{n+1}) =k^{j-1}d^j \\
k^0d^0 = (i_0 f, i_0 g,\dots ,i_{n}g)(i_1,\dots ,i_{n+1}) = (i_0 g,\dots ,i_{n}g) = \underbrace{g\sqcup\dots\sqcup g}_{n+1} = \BB(g)^n\\
k^nd^{n+1} = (i_0 f,\dots ,i_{n} f, i_{n} g)(i_0\dots i_{n}) =  (i_0 f,\dots ,i_{n}f) = \underbrace{f\sqcup\dots\sqcup f}_{n+1} = \BB(f)^n
\end{gather*}
This shows that $\{k^i\}_{i=0}^{\infty}$ defined above is indeed a cosimplicial homotopy between $\BB(f)$ and $\BB(g)$.
\end{proof}

\begin{cor}\label{cohomology_independent_of_object}
Let $\mathcal F:\CC\to \ab$ be a functor on a strongly connected $\CC$ with pair-wise coproducts. Then the cohomology groups
$$
\pi^n \mathcal F\BB (c):=H^n C \mathcal F \BB (c)
$$
are independent of $c\in\CC$.
\end{cor}

\begin{rmk}\label{strongly_connected_is_essential}
If a category $\CC$ is not strongly connected, $\BB$ can be quite far from being homotopically constant, as the following example shows (see \cite{amitsur2001simple}, \cite{rosenberg1960amitsur}). Let $k$ be a field and $\CC=k\mr{-Alg}$ be a category of commutative $k$-algebras and a coproduct is given by a tensor product over $k$.
Let $\mathcal F = U:k\mr{-Alg}\to k\mr{-Mod}$ be a forgetful functor, then for $A\in k\mr{-Alg}$ the (coaugmented) alternate sum complex
\begin{gather*}
U\BB(A): \ k\to A\xrightarrow{d} A\otimes_k A\xrightarrow{d} A\otimes_k A\otimes_k A\to\dots   \\
 d : (a_1\otimes \dots \otimes a_k)\mapsto (a_1\otimes \dots a_{i-1}\otimes 1 \otimes a_i\otimes \dots\otimes a_k)
\end{gather*}
is called the \textit{Amitsur complex} and its cohomology broadly depends on $A$. For example, for $A=k$, $U\BB(k) = k$ and the complex is contractible. But for $A$ being a finite dimensional extension of $k$ it can be shown (see \cite{amitsur2001simple}) that  $H^2 (U\BB(A))$ is the Brauer group of the corresponding extension.
\end{rmk}

Let $\mathcal F:\CC\to\ab$ be a functor. Below we will study cocycles and (co)homotopy groups of the cosimplicial object $\mathcal F\BB(c)$.
\begin{lem}\label{cofaces_induce_iso}
Cofaces \eqref{cofaces_and_codegenracies_in_the_standard_complex} induce isomorphisms on higher limits of $\mathcal F$:
$$
{\sf lim}^m \, \mathcal F(\sqcup^{n+1} c)\xrightarrow{\mathcal F(d^j)^*} {\sf lim}^m \, \mathcal F(\sqcup^{n+2} c)
$$
\end{lem}
\begin{proof}
 First two cofaces $i_1, i_2:c\to c\sqcup c$ in $\BB(c)$ are inducing isomorphisms
$$
{\sf lim}^n \,\mathcal F (c) \xrightarrow{\mathcal F(i_k)_*} {\sf lim}^n \, \mathcal F(c\sqcup c)
$$
by (3.6) in \cite{ivanov2015higher}. Modifying the proof of this lemma, one can see that the similar fact holds for all canonical inclusions $i_k:c\to \sqcup^n c$. This can be seen by considering a functor $\Phi_n: c\mapsto\sqcup^n c$ together with a natural transformation $i_k: \mr{id}\to \Phi_n$ such that for any $c'\in\CC$ the comma category  $(\Phi_n\downarrow c')$ is contractible. Now consider the diagrams
$$
\xym{
k<i+1:&&&&k\geq i+1:&&\\
\sqcup^{n+1}c\ar[rr]^{(i_0,\dots ,\hat i_{j},\dots ,i_{n+1})}&&\sqcup^{n+2}c && \sqcup^{n+1}c\ar[rr]^{(i_0,\dots ,\hat i_{j},\dots ,i_{n+1})}&&\sqcup^{n+2}c\\
&c\ar[ul]^{i_k}\ar[ur]_{i_k}& && &c\ar[ul]^{i_k}\ar[ur]_{i_{k+1}}&
}
$$
After applying $\mathcal F$ and ${\sf lim}^n$ diagonal arrows become isomorphisms, hence a horizontal arrow, which is a map, induced by coface, is an isomorphism too.
\end{proof}
Cocycles $\ZZ^n \mathcal F\BB(c)$ of the standard complex serve as a natural generalization of the functor of invariants \eqref{defin_of_invariants}:
\begin{lem}\label{defin_of_higher_invariants_are_cocycles_and}
For $c\in\CC$ the following formula holds:
\begin{equation}\label{higher_invariants_formula}
\ZZ^n \mathcal F\BB(c) =\{x\in \mathcal F(\sqcup^{n+1} c)| \forall c', \ \varphi_0,\dots,\varphi_{n+1}:c\to c' \\ \sum_{j=0}^{n+1} (-1)^j \mathcal F((\varphi_0,\dots ,\hat \varphi_j,\dots ,\varphi_{n+1}))(x)=0\}
\end{equation}
\end{lem}
\begin{proof}
By definition,
$$
\ZZ^n \mathcal F\BB(c) = \{x\in \mathcal F(\sqcup^{n+1})| \sum_{j=0}^{n+1}(-1)^j \mathcal F(d^j)(x)=0\}
$$
Let's denote the right hand side of \eqref{higher_invariants_formula} by ${\sf inv}^n\,\mathcal F(c)$.  The inclusion ${\sf inv}^n\,\mathcal F(c)\subset \ZZ^n \mathcal F\BB(c)$ is obvious. Now for any collection of maps $\varphi_0,\dots, \varphi_{n+1}: c\to c'$ there is a unique morphism $\Phi=(\varphi_0,\dots, \varphi_{n+1}):\sqcup^{n+2}c\to c'$ such that $\varphi_j = \Phi i_j$ and moreover $(\varphi_0,\dots ,\hat \varphi_j,\dots ,\varphi_{n+1}) = \Phi\circ (i_0,\dots ,\hat i_j ,\dots ,i_{n+1})$. Hence for $x\in  \ZZ^n \mathcal F\BB(c)$:
\begin{gather*}
\sum_{j=0}^{n+1}(-1)^j \mathcal F((\varphi_0,\dots,\hat \varphi_j,\dots ,\varphi_{n+1}))(x)= \\
\sum_{j=0}^{n+1}(-1)^j \mathcal F(\Phi)\circ \mathcal F((i_0,\dots ,\hat i_j,\dots ,i_{n+1}))(x) =\\
 \mathcal F(\Phi) (\sum_{j=0}^{n+1}(-1)^j \mathcal F((i_0,\dots ,\hat i_{j},\dots ,i_{n+1}))(x)) = \mathcal F(\Phi)(0)=0
\end{gather*}
and $x\in {\sf inv}^n\, \mathcal F(c)$
\end{proof}
The gap between the higher invariants ${\sf inv}^n$ and the higher limits of the functor $\mathcal F$ is given by the coboundaries of $\mathcal F\BB(c)$ as the following theorem shows and hence the standard complex \eqref{defin_of_standard_complex} can be used as a sort of resolution for computing ${\sf lim}^n\, \mathcal F$:
\begin{thm}\label{higher_limits_as_cohomology}
For strongly connected category $\CC$ with pair-wise coproducts and a functor $\mathcal F: \CC\to \mr{Ab}$ for any $c\in \CC$
\begin{equation}
{\sf lim}^n\, \mathcal F=\pi^n \mathcal F\BB(c)
\end{equation}
\end{thm}
\begin{proof}
By \eqref{cohomology_independent_of_object} the (co)homotopy groups of $\mathcal F\BB(c)$ are independent of $c$, in particular, a cochain complex $C \mathcal F\BB(-)$ is bounded below, has ${\sf lim}$-acyclic cohomology and there is a spectral sequence \eqref{lim_spectral_sequence}:
$$
E^{p,q}_1 = {\sf lim}^q\, \mathcal F\BB(\sqcup^p c) = {\sf lim}^q\, \mathcal F\BB(c)\ra {\sf lim} \,\pi^{p+q} \mathcal F\BB(c) = \pi^{p+q} \mathcal F\BB(c)
$$
The first page differential in this spectral sequence (which is acting horizontally) is a morphism, induced on ${\sf lim}^q$ by the differential of the alternate sum complex: $\sum_j (-1)^j \mathcal F(d^j)$. Each summand in this differential is an isomorphism by \eqref{cofaces_induce_iso} and hence the first and second page of the spectral sequence look like this:
\begin{gather*}
\small
\xym@R=4pt@C=5pt{
&&\vdots&&\vdots&&&&&&&\vdots&&\vdots\\
&&{\sf lim}^2\,\mathcal F(c)\ar[rr]^-{0}&&{\sf lim}^2\,\mathcal F(c\sqcup c)\ar[rr]^-{\cong}&&&\dots&&&&{\sf lim}^2\,\mathcal F(c)&&0&&0&&\dots\\
&&{\sf lim}^1\,\mathcal F(c)\ar[rr]^-{0}&&{\sf lim}^1\,\mathcal F(c\sqcup c)\ar[rr]^-{\cong}&&&\dots&&&&{\sf lim}^1\,\mathcal F(c)&&0&&0&&\dots\\
q\ar[u]&&{\sf lim}\,\mathcal F(c)\ar[rr]^-{0}&&{\sf lim}\,\mathcal F(c\sqcup c)\ar[rr]^-{\cong}&&&\dots&&q\ar[u]&&{\sf lim}\,\mathcal F(c)&&0&&0&&\dots\\
\ar[rrrr]&&&&&&&&&\ar[rrrr]&&&&&\\
E_1^{p,q}&\ar[uuuu]&p\ar[r]&&&&&&\ar@{-}[uuuuu]&E_2^{p,q}=E_{\infty}^{p,q}&\ar[uuuu]&p\ar[r]&
}
\end{gather*}
Assertion follows.
\end{proof}

\begin{defin}\label{defin_of_degree}
We say that the functor $\mathcal F:\CC\to \mr{Ab}$ has the \textit{degree} $\mr{deg}\,\mathcal F\leq n$ if $Q(\mathcal F\BB(c))^k=0$ for all $k>n$ for some $c\in\CC$.
\end{defin}
This definition of the degree is a generalization (see \cite{pirashvili1982spectral}) of the usual notion of the degree of a polynomial functor between abelian categories \cite{eilenberg1954groups}. We will sketch the (dual version of) main ideas from \cite{pirashvili1982spectral}.

For a category $\CC$ let $\CC_{(1)}$ be a category of splitting monomorphisms of the form $c\to c\sqcup c'$, iteratively $\CC_{(k)}=(\CC_{(k-1)})_{(1)}$. Given a functor $\mathcal F:\CC\to\mr{Ab}$ its \textit{coderivative} is defined as
$$
\mathcal F_{(1)} (c\to c\sqcup c')={\sf coker}\, \{\mathcal F(c)\to \mathcal F(c\sqcup c')\}
$$
Similarly the higher orders coderivatives of $\mathcal F$ are defined. Then the dual version of Proposition 1.7 of \cite{pirashvili1982spectral} holds:

\begin{prop}\label{pirashvili_degree_equals_our_degree}
Let $\mathcal F$ be a functor such that $\mathcal F_{(k)} = 0$ for some $k$. Then $\mr{deg}\, \mathcal F \leq k-1$.
\end{prop}
\begin{proof}
For a cosimplical object $X$ define $k$-cubes $c_k(X)$ iteratively as
$$
c_0(X) = X_0, \ c_{k+1} (X) = c_k(X)\xrightarrow{d^1} c_k (\mr{Dec}\, X)
$$

Then for $X=\mathcal F\BB(c), \  \mathcal F_{(k-1)} (c_{k-1}(X))={\sf coker}\, \{\mathcal F_{(k-1)}(c_{k-1}(X))\xrightarrow{d^1} \mathcal F_{(k-1)}(c_{k-1}(\mr{Dec}\, X))\}$ and from \eqref{iterative_description_of_moore_complex} and the induction we get that
$$
(QX)^k = \mathcal F_{(k)}(c_k(X))
$$.
\end{proof}
The degree functor $\mr{deg}$ behaves in a predictable way with a tensor product of functors :
\begin{thm}\label{degree_of_tensor_product}
Let $\mathcal F$ and $\GG$ be functors of degrees $\leq n$ and $\leq m$ respectively. Then their tensor product $\FF\otimes\GG$ has degree $\leq n+m-1$.
\end{thm}
\begin{proof}
For a given split monomorphism $f:c\to c'$ in $\CC$ the map $(\FF\otimes\GG)(f)$ divides into composition of two split monomorphisms
$$
(\FF(c)\otimes\GG(c))\xrightarrow{\mr{id}\otimes \GG(f)}(\FF(c)\otimes\GG(c'))\xrightarrow{\FF(f)\otimes\mr{id}} (\FF(c')\otimes\GG(c'))
$$
hence coderivative of the tensor product splits as
$$
(\FF\otimes\GG)_{(1)}(f)=\FF(c)\otimes\GG_{(1)}(f)\oplus\FF_{(1)}(f)\otimes\GG(c')
$$
By iterating this formula we get
$$
(\FF\otimes\GG)_{(k)}=\bigoplus_{i+j=n}s^i\FF_{(j)}\otimes t^j\GG_{(i)}
$$
Result now follows from this formula and Proposition \ref{pirashvili_degree_equals_our_degree}.
\end{proof}
\section{K\"unneth theorem}
We can use the fact that ${\sf lim}^n\,\FF$ can be expressed as cohomology groups of a well-understood complex to determine the higher limits of a tensor product of functors, using a K\"unneth-type spectral sequence as in (6.8) of \cite{quillen}. For later use in \eqref{fr-codes_section}  we will expand our universe of functors and describe the K\"unneth formula in this generalized setting.

As in \cite{ivanov2015higher}, let $\mr{Mod}_{\sf r}$  denote the category of pairs $(R,M)$, where $R$ is a ring and $M$ is a right $R$-module. Morphisms are pairs $(f,\varphi): (R,M)\to (S,N)$ consisting of ring homomorphism $f:R\to S$ and $R$-linear map $\varphi: M\to N$, where $R$ acting on $N$ through $f$. There is a natural projection $\mr{Mod}_{\sf r}\to\mr{Ring}$. Similarly, $\mr{Mod}_{\sf l}$ will denote the category of left modules over arbitrary rings.
\begin{defin}[\cite{ivanov2015higher}, (3.2)]\label{defin_of_O-module}
Let $\OO:\CC\to \mr{Ring}$ be a $\mr{Ring}$-valued functor. Then the \textit{right $\OO$-module} $\FF:\CC\to\mr{Mod}$ is a functor, such that the following diagram commutes:
$$
\xym{
	&\mr{Mod}_{\sf r}\ar[d]\\
	\CC\ar[ur]^{\FF}\ar[r]^{\OO}&\mr{Ring}
}
$$
Definition of the \textit{left $\OO$-module} is completely symmetric.
\end{defin}
Note that the higher limits ${\sf lim}^{\bullet}\,\FF$ have a structure of a graded module over graded ring ${\sf lim}^{\bullet}\,\OO$. For the right $\OO$-module $\FF$ and the left $\OO$-module $\GG$ their \textit{tensor product over $\OO$} is defined as a functor:
$$
\FF\otimes_{\OO}\GG:\CC\to \mr{Ab}, \ \FF\otimes_{\OO}\GG(c) = \FF(c)\otimes_{\OO(c)}\GG(c)
$$

\begin{thm}\label{Kunneth_thm}
Let $\OO:\CC\to \mr{Ring}$ be a functor such that ${\sf lim}^{\bullet}\OO$ is of finite global dimension. For a right $\OO$-module $M$ and left $\OO$-module $N$ of finite degree, such that $N(c)$ is a flat $\OO (c)$-module for all $c\in\CC$ there is a second quadrant spectral sequence
$$
E^2_{p,q} = {\sf Tor}_p^{{\sf lim}^{\bullet}\,\OO}({\sf lim}^{\bullet}\,M,{\sf lim}^{\bullet}\,N)_q\ra {\sf lim}^{\bullet}\,M\otimes_{\OO} N
$$
\end{thm}
\begin{proof}
The proof is a direct combination of the cosimplicial version of Theorem 6 of \cite{quillen}  and \eqref{higher_limits_as_cohomology}. Fix $c\in\CC$ and consider the projective resolution $P_{\bullet}$ of $M\BB(c)$ over $\OO\BB(c)$ such that $\pi^{\bullet}P_i$ are free $\pi^{\bullet}\OO\BB(c)$ modules for all $i$. The resolution $P_{\bullet}$ can be constructed in a way that $\pi^{\bullet}P_{\bullet}$ is a free resolution of ${\sf lim}^{\bullet}\,M$ over ${\sf lim}^{\bullet}\,\OO$ and hence there is an isomorphism of graded abelian groups
$$
\pi^{\bullet}P_i\otimes_{{\sf lim}^{\bullet}\,\OO}{\sf lim}^{\bullet}\,N\cong \pi^{\bullet}(P_i\otimes_{\OO\BB(c)}N\BB(c))
$$
\end{proof}
Applying the Moore chain complex functor $Q$ horizontally to the cosimplicial chain complex $D =P_{\bullet}\otimes_{\OO\BB(c)} N\BB(c)$ and switching to the homological notation, we obtain a second quadrant double complex. Further argument is standard. Consider two spectral sequences, associated with $D$:
\begin{itemize}

\item $E^2_{p,q}=H_p^h H_q^v D = H_p Q({\sf Tor}^{\OO\BB(c)}(M\BB(c),N\BB(c)))$
Provided $N\BB(c)$ is free as $\OO\BB(c)$-module, only the bottom line is nontrivial on the second page and the spectral sequence converges to ${\sf lim}^{\bullet}\,M\otimes_{\OO} N$
\item $E^2_{p,q}=H_q^v H_p^h D = H_q (\pi^{\bullet}P_p\otimes_{{\sf lim}^{\bullet}\,\OO}{\sf lim}^{\bullet}\,N) = ({\sf Tor}_p^{{\sf lim}^{\bullet}\,\OO}({\sf lim}^{\bullet}\,M,{\sf lim}^{\bullet}\,N))_q$

Since ${\sf Tor}$-functors vanish above the certain line this spectral sequence converges to the same limit, as the first one.
\end{itemize}

\section{$\f\r$-codes}\label{fr-codes_section}

We denote by $\mr{Pres}$ the category whose objects are all presentations $c: F \twoheadrightarrow G$ and morphisms are commutative squares
\begin{equation}\label{eq_mor_pres}
\xym{
F\arrow[r]^{\tilde  \varphi }\arrow@{->>}[d]^{c} & F' \arrow@{->>}[d]^{c'} \\
G\arrow[r]^{\varphi} & G'
}
\hspace{3cm}
(\varphi,\tilde \varphi):c\to c'.
\end{equation}
For each group $G$ the category $\mr{Pres}(G)$ is a subcategory of $\mr{Pres}.$ Then for a functor $$\mathcal F:\mr{Pres}\to \mr{Ab}$$ and any $i\geq 0$ we have a map  
$$G\mapsto \underset{\mr{Pres}(G)}{{\sf lim}^i} \  \mathcal F  .$$ 
Here $c: F \twoheadrightarrow G$ 
can be considered as an object of $\mr{Pres}(G)$ and 
we can take $\BB(c)$ that 
we will denote by 
$\BB_G(c)$ 
in order to emphasize that we take it in the category $\mr{Pres}(G)$ 
but not in the whole category $\mr{Pres}.$
By Theorem \ref{higher_limits_as_cohomology} we have an isomorphism $$ \underset{\mr{Pres}(G)}{{\sf lim}^i} \  \mathcal F \cong H^i \mathcal F \BB_G (c).$$ 
Moreover any morphism \eqref{eq_mor_pres} in the category $\mr{Pres}$ gives a morphism of cosimplicial objects 
$$\BB_{\varphi,\tilde \varphi} :\BB_G(c) \to \BB_{G'}(c').$$ Then the morphism $(\varphi, \tilde  \varphi)$ induces a homomorphism  
\begin{equation}\label{eq_map_lim}
\underset{\mr{Pres}(G)}{{\sf lim}^i} \  \mathcal F \longrightarrow \underset{\mr{Pres}(G')}{{\sf lim}^i} \  \mathcal F.
\end{equation}

\begin{lem}\label{lemma_functoriality}
The homomorphism \eqref{eq_map_lim} depends only on $\varphi$ and does not depend on the choice of presentations and $\tilde \varphi.$ Moreover, these homomorphisms define a functor 
$$\underset{\mr{Pres}(G)}{{\sf lim}^i}\FF : \mr{Gr} \longrightarrow \mr{Ab}.$$
\end{lem}
\begin{proof}
Assume that we have two presentations $c_i: F_i \twoheadrightarrow G,$ $i=1,2$  for $G,$ two presentations  $c_i':F_i'\twoheadrightarrow G'$ for $G'.$ Assume also that we have two morphisms $(\varphi, \tilde \varphi_i ):c_i \to c_i'$ in $\mr{Pres}.$ Consider the presentations $ c_1*c_2 : F_1*F_2\twoheadrightarrow G $ and $ c_1'*c_2' : F_1'*F_2'\twoheadrightarrow G',$ the morphism $(\varphi,\tilde \varphi_1*  \tilde \varphi_2):c_1*c_2 \to c_1'*c_2'$ and the commutative diagram
$$
\xym{
\BB_G(c_1)\arrow[rr]^{\BB_{\varphi,\tilde \varphi_1}}\arrow[d]  && \BB_{G'}(c_1')\arrow[d]  \\
\BB_G(c_1* c_2)\arrow[rr]^{\BB_{\varphi,\tilde \varphi_1* \tilde  \varphi_2}} && \BB_{G'}(c_1'* c_2')\\ 
\BB_G(c_2)\arrow[rr]^{\BB_{\varphi,\tilde \varphi_2}} \arrow[u] && \BB_{G'}(c_2') \arrow[u] 
 }.
$$ 
By Theorem \ref{hatC_homotopy_constant} the vertical arrows induce isomorphisms on $H^i \mathcal F \BB(-).$ The assertion follows. 
\end{proof}

The group ring functor $\Z[-]:\mr{Pres}\to\mr{Ring}, \ (F\twoheadrightarrow G)\mapsto \Z[F]$ has two functorial ideals ($\Z[F]$-modules in sense of Definition \ref{defin_of_O-module}) $\f$ and $\r$ defined as
$$
\f (F\twoheadrightarrow G)= {\sf ker}\{\Z[F]\to\Z\}, \ \r (F\twoheadrightarrow G)= {\sf ker}\{\Z[F]\to\Z[G]\}
$$
  
\begin{defin}
The $\Z[F]$-module ${\bf c}:\mr{Pres}\to\mr{Ab}$ is called an \textit{$\f\r$-code}, if it is a functorial ideal of $\Z[F]$, formed by products of the ideals $\f$ and $\r$, their sums and intersections. 

Usually we consider ${\bf c}$ as a functor from $\mr{Pres}(G)\to \mr{Ab}$ for a fixed $G,$ limits always are taken over $\mr{Pres}(G).$ We need it to be defined on the category $\mr{Pres}$ only for the functors
$$ {\sf lim}^i\ {\bf c} : \mr{Gr} \longrightarrow \mr{Ab}$$
to be well-defined. (see Lemma \ref{lemma_functoriality}).
\end{defin}

The notion of degree \eqref{defin_of_degree} seems to be a reasonable invariant of $\f\r$-code for the estimation of its ${\sf lim}^{\bullet}$-dimension, since the Moore chain complex functor $Q$ is exact and the property of a functor being a degree $\leq k$ is closed under extensions. But already $\f$ itself has an infinite degree, although it is \textit{$\Z[F]$-additive} (i.e. of degree one with respect to $\Z[F]$), as shown in \cite{ivanov2017fr}. But since all $\f\r$-codes are subfunctors of $\f$, this difficulty can be overcame by introducing the following notion:
\begin{defin}\label{defin_of_f-degree}
An \textit{$\f$-degree} of an $\f\r$-code $\c$ is a degree of the quotient $\f/\c$.
\end{defin}
Since $\f$ has trivial limits, it is straightforward that if $\mr{deg}^{\f}\,\c\leq n$ then ${\sf lim}^{i}\,\c=0$ for $i>n+1$.
\begin{thm}\label{fr-codes_are_finite}
Every (finite) $\f\r$-code $\c$ has a finite $\f$-degree and hence only a finite number of the non-zero higher limits.
\end{thm}
\begin{proof}
Let $n$ be a minimal power of $\r$ such that $\r^n\subset \c$, then we have an epimorphism $\f/\r^n\twoheadrightarrow \f/\c$ which induces a surjection on the level of cochain complexes:
$$
Q\frac{\f}{{\bf r}^n}\BB\twoheadrightarrow Q\frac{\f}{{\bf c}}\BB
$$
and hence it is sufficient to prove finiteness of $\f/\r^n$. The sequence of the short exact sequences
\begin{equation}\label{fr_short_exact_sequence}
\r^n/\r^{n+1}\hookrightarrow \f/\r^{n+1}\twoheadrightarrow \f/\r^n
\end{equation}
starts with a constant functor $\f/\r=\g={\sf ker}\,\{\Z[G]\to\Z\}$ and the problem is reduced to the functors $\r^n/\r^{n+1} = (\r/\r^2)^{\otimes_{\Z[F]} n}$ (see Lemma \ref{lemmaprod}). Covering this tensor product by the tensor product over $\Z$ and applying Theorem \ref{degree_of_tensor_product}, only the case $n=1$ need to be shown. Note that $\r/\r^2$ is a free $\Z[G]$-module with basis formed by elements $r-1, \ r\in R$, see \cite{gruenberg2006cohomological}, hence a natural map $R_{ab}\to \r/\r^2, \ r\mapsto r-1$ factors through $\Z [G] \otimes R_{ab}\to \r/\r^2$ and this map is an isomorphism.

Finally, the functor $R_{ab} = \r/\f\r$ has a finite degree, since it is embedded in $\f/\f\r=\f\otimes_{\Z [F]}\Z [G]$ which is an additive functor. Indeed (see also \cite{weibel1995introduction}):
$$
\f(F*F') \otimes_{\Z [F*F']}\Z G= (\f(F)\otimes_{\Z [F]}\Z [F*F']\oplus\f(F')\otimes_{\Z [F']}\Z [F*F'])\otimes_{\Z [F*F']} \Z [G]
$$

$$
= \f(F)\otimes_{\Z [F]} \Z [G]\oplus \f(F')\otimes_{\Z [F']} \Z [G]
$$
which concludes the proof.
\end{proof}

\section{Dictionary}
In this section, we give a dictionary for all codes written on ${\bf fr}$-language which consist of words with length $\leq 3$. If one can not find a code in our table, this means that either it has trivial translation, i.e. all ${\sf lim}^i=0$, or has the same translation as its mirror image, which is in our dictionary. For example, the codes ${\bf rf+ffr}$ and ${\bf fr+rff}$ have the same translations. As mention in Introduction, by {\it translation} we mean a description of the functors ${\sf lim}^i\left({\bf fr}-\mathrm{code}\right)$, $i\geq 1$, ${\bf fr-}\mathrm{codes}$ viewed as functors from the category of free group presentations to the category of abelian groups.

We will omit the translation of simple codes given in \cite{ivanov2017fr}, like $\bf rr+fff$, or $\bf rr+frf,\ rrf+frr$.

 In construction of the dictionary, we will use the following statements.
\begin{lem}[Lemma 5.9 in \cite{ivanov2017fr}]\label{lemmaprod}
Let ${\bf a'\subset a,\ b'\subset b}$ be ideals of $\mathbb Z[F]$ and ${\sf Tor}(\mathbb Z[F]/{\bf a}, \mathbb Z[F]/{\bf b})=0$, then there is a natural isomorphism
$$
\frac{\bf a}{\bf a'}\otimes_{\Z[F]}\frac{\bf b}{\bf b'}=\frac{\bf ab}{\bf ab'+a'b}.
$$
\end{lem}

\begin{lem}\label{tensor_lemma}
For any functor $\mathcal F(F,R)$ and a non-constant functor $\mathcal H(F)$, which depends only on $F$, ${\sf lim}^i\mathcal F\otimes \mathcal H=0,\ i\geq 0$.
\end{lem}
Similarly one can show (see \cite{ivanov2017fr}) that, for a $\bf fr$-code with all words started with $\bf f,$ all limits are zero.
\begin{lem}[Lemma 6 in \cite{mikhailov2019dimension}]\label{constant_functor_lemma}
Let $\mathcal F$ be a constant functor. Then any subfunctor $\mathcal G \hookrightarrow \mathcal F$ and any epimorphic image $\mathcal F \twoheadrightarrow \mathcal H$ are constant functors.
\end{lem}

We will also use the spectral sequence \ref{lim_spectral_sequence}, especially applied to the 4-term complexes. For convenience, lets reformulate the statement about convergence of the spectral sequence \ref{lim_spectral_sequence} in a more explicit form. Let  $\mathcal F^\bullet$ be a complex of functors $\mr{Pres}(G)\to \mr{Ab}$
$$\dots \to \mathcal F^{n-1} \to \mathcal F^n \to \mathcal F^{n+1} \to \dots $$
Assume that $\mathcal F^\bullet$ is bounded below (i.e.  $\mathcal F^n=0 $ for $n<\!<0$) and that $H^n(\mathcal F^\bullet) $ is constant for any $n.$ Then there exists a converging  spectral sequence $E$ with differentials  $$d^r : E_r^{i,j} \longrightarrow E_{r}^{i+r,j-r+1}$$  such that
$$E_1^{i,j}={\rm lim}^j \mathcal F^i \Rightarrow H^{i+j}(\mathcal F^\bullet).$$

Now we proceed to the computations.
\vspace{.5cm}\noindent

{\bf rfr+frf}:\ Tensoring the short exact sequence $\frac{\bf r}{\bf fr}\hookrightarrow \frac{\bf f}{\bf fr}\twoheadrightarrow {\bf g}$ by $-\otimes\frac{\bf f}{\bf fr+rf}$ and taking the group homology $H_i(G,-)$, we get the long exact sequence
\begin{multline}\label{les1}
H_1\left(G, \frac{\bf f}{\bf fr}\otimes \frac{\bf f}{\bf fr+rf}\right)\to H_2\left(G, \frac{\bf f}{\bf fr+rf}\right)\to
\frac{\bf r}{\bf fr}\otimes_{\mathbb Z[G]}\frac{\bf f}{\bf fr+rf}\to\\ \to\frac{\bf f}{\bf fr}\otimes_{\mathbb Z[G]}\frac{\bf f}{\bf fr+rf}\twoheadrightarrow {\bf g}\otimes_{\mathbb Z[G]}\frac{\bf f}{\bf fr+rf}.
\end{multline}
Here we used the property that, for any $G$-module $M$, there is a natural isomorphism $H_1(G, {\bf g}\otimes M)=H_2(G, M)$.
Since $\f/\f\r$ is a free $\Z[G]$-module, $\frac{\bf f}{\bf fr}\otimes \frac{\bf f}{\bf fr+rf}$ is weak projective and hence $H_i\left(G, \frac{\bf f}{\bf fr}\otimes \frac{\bf f}{\bf fr+rf}\right)=0,\ i\geq 1$. By Lemma \ref{lemmaprod},
\begin{align*}
& \frac{\bf r}{\bf fr}\otimes_{\mathbb Z[G]}\frac{\bf f}{\bf fr+rf}=\frac{\bf rf}{\bf rfr+frf},\
& \frac{\bf f}{\bf fr}\otimes_{\mathbb Z[G]}\frac{\bf f}{\bf fr+rf}=\frac{\bf ff}{\bf ffr+frf}.
\end{align*}
And $\frac{\bf ff}{\bf ffr+frf}$ has trivial limits by Lemma \ref{tensor_lemma}. From a spectral sequence of Proposition \ref{lim_spectral_sequence} applied to a four-term exact sequence \eqref{les1}, it can be seen that

$$
{\sf lim}\frac{\bf rf}{\bf rfr+frf} = {\sf lim}^1({\bf rfr+frf})={\sf lim}\ H_2\left(G, \frac{\bf f}{\bf fr+rf}\right)
$$

and there is a short exact sequence
$$
{\sf lim}^1\ H_2\left(G, \frac{\bf f}{\bf fr+rf}\right)\hookrightarrow {\sf lim}^2({\bf rfr+frf})\twoheadrightarrow {\bf g}^2\otimes_{\mathbb Z[G]}\bf g.
$$
Here $\g\otimes_{\Z[G]} \g^2={\sf lim}(\g\otimes_{\Z[G]} \frac{\f}{\f\r+\r\f})={\sf lim}^1({\bf rf+ffr})$. To determine the ${\sf lim}$ and ${\sf lim}^1$ of $H_2(G,\frac{\bf f}{\bf fr+rf})$, consider the short exact sequence
$$
{\bf \frac{ff}{fr+rf}}\hookrightarrow {\bf \frac{f}{fr+rf}}\twoheadrightarrow {\bf \frac{f}{ff}}
$$
and the associated homology long exact sequence
\begin{equation}\label{les2}
H_3(G)\otimes F_{ab}\to H_2\left(G, {\bf g}\otimes_{\mathbb Z[G]}{\bf g}\right)\to H_2\left(G, \frac{\bf f}{\bf fr+rf}\right)\to H_2(G)\otimes F_{ab}\to
H_1\left(G, {\bf g}\otimes_{\mathbb Z[G]}{\bf g}\right)
\end{equation}

Any map from $H_n(G)\otimes F_{ab}$ to a constant functor (which depends only on $G$) factors through $H_n(G)\otimes G_{ab}$. This follows from elementary properties of colimits (see \cite{ivanov2018colimits}), namely
from ${\sf colim} (H_n(G)\otimes F_{ab})=H_n(G)\otimes G_{ab}$. Therefore, after truncating \eqref{les2} and applying Proposition \ref{lim_spectral_sequence} together with Lemma \ref{constant_functor_lemma} to it, we obtain
$$
{\sf lim}^1({\bf rfr+frf})={\sf coker}\{H_3(G)\otimes G_{ab}\to H_2\left(G, {\bf g}\otimes_{\mathbb Z[G]}{\bf g}\right)\}.
$$
and
$$
{\sf lim}^1\ H_2\left(G, \frac{\bf f}{\bf fr+rf}\right)={\sf im}\{H_2(G)\otimes G_{ab}\to H_1\left(G, {\bf g}\otimes_{\mathbb Z[G]}{\bf g}\right)\}.
$$
Hence, there is a short exact sequence
$$
{\sf im}\{H_2(G)\otimes G_{ab}\to H_1\left(G, {\bf g}\otimes_{\mathbb Z[G]}{\bf g}\right)\}\hookrightarrow {\sf lim}^2({\bf rfr+frf})\twoheadrightarrow {\bf g}^2\otimes_{\mathbb Z[G]}\bf g.
$$

\vspace{.5cm}\noindent
{\bf rr+frf+rff:} Consider the Gr\"unberg resolution which consists of free $\mathbb Z[G]$-modules:
$$
\dots \to \frac{\bf fr}{\bf frr}\to \frac{\bf r}{\bf rr}\to \frac{\bf f}{\bf fr}
$$
Tensoring it with $G_{ab}=\frac{\bf f}{\bf r+ff}$ over $\mathbb Z[G]$, we obtain the complex
$$
\dots \to \frac{\bf fr}{\bf frr}\otimes_{\mathbb Z[G]}\frac{\bf f}{\bf r+ff}\to \frac{\bf r}{\bf rr}\otimes_{\mathbb Z[G]}\frac{\bf f}{\bf r+ff}\to \frac{\bf f}{\bf fr}\otimes_{\mathbb Z[G]}\frac{\bf f}{\bf r+ff}
$$
which can be written, by Lemma \ref{lemmaprod} as
$$
\dots \to \frac{\bf frf}{\bf frr+frff}\to \frac{\bf rf}{\bf rr+rff}\to \frac{\bf ff}{\bf fr+fff}
$$
Hence, there is a natural isomorphism
$$
H_2(G, G_{ab})=\frac{\bf rf\cap (fr+fff)}{\bf rr+rff+frf}.
$$
For two ideals $I,J\subset \f$ there is a short exact sequence:
$$
\frac{\f}{I\cap J}\hookrightarrow 
\frac{\bf f}{I}\oplus \frac{\bf f}{J}\twoheadrightarrow \frac{\f}{I+J}
$$
And hence we get the following 4-term exact sequence
$$
H_2(G,G_{ab})\hookrightarrow \frac{\bf f}{\bf rr+rff+frf}\to
\frac{\bf f}{\bf fr}\oplus \frac{\bf f}{\bf fr+fff}\twoheadrightarrow G_{ab}\otimes G_{ab}.
$$
From the associated spectral sequence we obtain the identifications
\begin{align*}
& {\sf lim}^1({\bf rr+frf+rff})=H_2(G, G_{ab}),\\
& {\sf lim}^2({\bf rr+frf+rff})=G_{ab}\otimes G_{ab},\\
& {\sf lim}^i({\bf rr+frf+rff})=0,\ i\geq 3.
\end{align*}
Now observe that, the statement written in Introduction, that the transformation of the ${\bf fr}$-codes
$$
{\bf rr+frf}\rightsquigarrow {\bf rr+frf+rff}
$$
induces the natural transformation of functors
$$
H_3(G)={\sf lim}^1({\bf rr+frf})\rightsquigarrow H_2(G,G_{ab})={\sf lim}^1({\bf rr+frf+frr})
$$
follows immediately from the identifications 
$$
H_3(G)=\frac{{\bf rf}\cap {\bf fr}}{\bf rr+frf}\to \frac{\bf rf\cap (fr+fff)}{\bf rr+frf+rff}. 
$$

\vspace{.5cm}\noindent
{\bf rrf+rfr+frr:} Taking the tensor product over $\Z[G]$ of the Gr\"unberg resolution with $H_2(G)=\frac{\bf r\cap ff}{\bf fr+rf}$ and $\frac{\bf r}{\bf fr+rf}$ respectively, we obtain
\begin{align*}
& H_2(G,H_2(G))=\frac{\bf r(r\cap ff)\cap (ffr+frf)}{\bf fr(r\cap ff)+rrf+rfr},\\
& H_2\left(G, \frac{\bf r}{\bf fr+rf}\right)=\frac{\bf rr\cap (ffr+frf)}{\bf rrf+rfr+frr}.
\end{align*}
Since ${\sf lim}^1 {\bf rr}={\sf lim}^1({\bf ffr+frf})=0,$
\begin{equation}\label{z11}
{\sf lim}^1 ({\bf rrf+rfr+frr})={\sf lim}\ H_2\left(G, \frac{\bf r}{\bf fr+rf}\right).
\end{equation}
The natural map $H_2(G,H_2(G))\to H_2\left(G, \frac{\bf r}{\bf fr+rf}\right)$ is injective. Indeed, the above terms can be decomposed as
$$
0\to H_2(G)\otimes H_2(G)\to H_2(G, H_2(G))\to {\sf Tor}(G_{ab}, H_2(G))\to 0
$$
and
$$
0\to H_2(G)\otimes \frac{\bf r}{\bf fr+rf}\to H_2\left(G, \frac{\bf r}{\bf fr+rf}\right)\to {\sf Tor}\left(G_{ab}, \frac{\bf r}{\bf fr+rf}\right)\to 0.
$$
Using the fact that $\frac{\bf r}{\bf r\cap ff}={\sf coker}\{H_2(G)\hookrightarrow \frac{\r}{\f\r+\r\f}\}$ is torsion-free (since it is a subgroup of $\f/\f\f = F_{ab}$), we see that the natural map $H_2(G)\otimes H_2(G)\to H_2(G)\otimes \frac{\bf r}{\bf fr+rf}$ is injective and
${\sf Tor}(G_{ab}, H_2(G))\to {\sf Tor}\left (G_{ab}, \frac{\bf r}{\bf fr+rf}\right)$ is an isomorphism. The natural map $H_1(G, H_2(G))\to H_1\left(G, \frac{\bf r}{\bf fr+rf}\right)$ is also injective, by the same reason. Hence, we have the following short exact sequence
\begin{equation}\label{z12}
0\to H_2(G, H_2(G))\to H_2\left(G, \frac{\bf r}{\bf fr+rf}\right)\to H_2(G)\otimes \frac{\bf r}{\bf r\cap ff}\to 0.
\end{equation}
It follows that
$$
{\sf lim}^1H_2(G, \frac{\bf r}{\bf fr+rf}) = {\sf lim}^1 (H_2(G)\otimes \frac{\bf r}{\bf r\cap ff})
$$
To compute the latter one, we use K\"unneth theorem \ref{Kunneth_thm}, which in this case degenerates to
$$
{\sf lim}^1(H_2(G)\otimes \frac{\bf r}{\bf r\cap ff}) = H_2(G)\otimes {\sf lim}^1 \frac{\bf r}{\bf r\cap ff}  = H_2(G)\otimes G_{ab}
$$
Applying Proposition \ref{lim_spectral_sequence} to the 4-term exact sequence
$$
0\to {\sf Tor}(G_{ab}, H_2(G))\to H_2(G)\otimes \frac{\bf r}{\bf r\cap ff}\to H_2(G)\otimes \frac{\bf f}{\bf ff}\to H_2(G)\otimes G_{ab}\to 0,
$$
we obtain the following description
\begin{align*}
& {\sf lim}\ H_2(G)\otimes \frac{\bf r}{\bf r\cap ff}={\sf Tor}(G_{ab}, H_2(G)), \ &
& {\sf lim}^1\ H_2(G)\otimes \frac{\bf r}{\bf r\cap ff}=H_2(G)\otimes G_{ab}.
\end{align*}
The isomorphism (\ref{z11}) and the exact sequence (\ref{z12}) now imply that there exists the following natural short exact sequence
$$
0\to H_2(G, H_2(G))\to {\sf lim}^1({\bf rrf+rfr+frr})\to {\sf Tor}(G_{ab}, H_2(G))\to 0
$$

In order to understand ${\sf lim}^2({\bf rrf+rfr+frr}),$ consider the spectral sequence applied to the 4-term sequence
$$
0\to H_2\left(G, \frac{\bf r}{\bf rf+fr}\right)\to \frac{\bf rr}{\bf rrf+rfr+frr}\to \frac{\bf fr}{\bf frf+ffr}\to \frac{\bf fr}{\bf rr+frf+ffr}\to 0.
$$
Putting the values of ${\sf lim}^i({\bf rr+frf+ffr}), {\sf lim}^1 H_2\left(G, \frac{\bf r}{\bf rf+fr}\right)$ into the cells of the spectral sequence and noting that
${\sf lim}^2{\bf rr}={\bf g\otimes g},\ {\sf lim}^i{\bf rr}=0,\ i\neq 2,$ we obtain the following diagram which gives a description of ${\sf lim}^2({\bf rrf+rfr+frr})$ as a functor glued from three pieces
$$
\xyma{H_2(G)\otimes G_{ab}\ar@{>->}[d]\\ {\sf lim}^1\frac{\bf rr}{\bf rrf+rfr+frr}\ar@{>->}[r]\ar@{->>}[d] & {\sf lim}^2({\bf rrf+rfr+frr})\ar@{->>}[r] & {\sf ker}\{{\bf g}\otimes {\bf g}\to G_{ab}\otimes G_{ab}\}\\ H_2(G,G_{ab})}
$$

\vspace{.5cm}\noindent
{\bf ffr+rff:} First observe that,
\begin{equation}\label{xx1}
{\sf lim}^i({\bf ffr+rff+ffff})=0, i\geq 0.
\end{equation}
This follows from the isomorphism
$$
G_{ab}\otimes F_{ab}\otimes G_{ab}=\frac{\bf f}{\bf r+ff}\otimes \frac{\bf f}{\bf ff}\otimes \frac{\bf f}{\bf r+ff}=
$$
$$
=\frac{\bf f}{\bf r+ff}\otimes_{Z[F]} \frac{\bf f}{\bf ff}\otimes_{\Z[F]} \frac{\bf f}{\bf r+ff}=\frac{\bf ff}{\bf rf+fff}\otimes \frac{\bf f}{\bf r+ff}=\frac{\bf fff}{\bf ffr+rff+ffff}.
$$
Consider the following exact sequence
$$
\frac{\bf ff}{\bf r\cap ff}\otimes_{\mathbb Z[G]} \frac{\bf ff}{\bf r\cap ff}\to \frac{\bf fff}{\bf ffr+rff}\to \frac{\bf fff}{\bf ffr+rff+ffff}\to 0.
$$
The left hand term is ${\bf g}^2\otimes_{\mathbb Z[G]}{\bf g}^2$. Since an epimorphic image of a constant functor is a constant functor, (\ref{xx1}) implies that
\begin{align*}
& {\sf lim}^1({\bf ffr+rff})=\frac{{\bf g}^2\otimes_{\mathbb Z[G]}{\bf g^2}}{\sim}, \ &
& {\sf lim}^i({\bf ffr+rff})=0,\ i\geq 2,
\end{align*}
where $\frac{{\bf g}^2\otimes_{\mathbb Z[G]}{\bf g^2}}{\sim}$ is the image of the left hand map in the last exact sequence, i.e.
$$
\frac{{\bf g}^2\otimes_{\mathbb Z[G]}{\bf g^2}}{\sim}=\frac{\bf ffff}{\bf (ffr+rff)\cap ffff}.
$$

\vspace{.5cm}\noindent
{\bf rr+ffr+rff:} Define one more quotient of ${\bf g}^2\otimes_{\mathbb Z[G]}{\bf g^2}$ as follows\footnote{Observe that, there exists a natural exact sequence $${\sf Tor}(G_{ab},G_{ab})\to {\bf g}^2\otimes_{\mathbb Z[G]}{\bf g^2}\to \frac{{\bf g}^2\otimes_{\mathbb Z[G]}{\bf g^2}}{\approx}\to 0.$$}:
$$
\frac{{\bf g}^2\otimes_{\mathbb Z[G]}{\bf g^2}}{\approx}:=\frac{\bf (r+ff)^2}{\bf rr+rff+ffr}.
$$
There is a natural epimorphism
$\frac{{\bf g}^2\otimes_{\mathbb Z[G]}{\bf g^2}}{\sim}\twoheadrightarrow \frac{{\bf g}^2\otimes_{\mathbb Z[G]}{\bf g^2}}{\approx}.$ The short exact sequence
$$
0\to \frac{{\bf g}^2\otimes_{\mathbb Z[G]}{\bf g^2}}{\approx}\to \frac{\bf ff}{\bf rr+ffr+rff}\to \frac{\bf ff}{\bf (r+ff)^2}\to 0
$$
implies that
$$
{\sf lim}^2({\bf rr+ffr+rff})={\sf lim}^2({\bf (r+ff)}^2)
$$
and there is an exact sequence
$$
0\to \frac{{\bf g}^2\otimes_{\mathbb Z[G]}{\bf g^2}}{\approx}\to {\sf lim}^1(\bf rr+ffr+rff)\to {\sf lim}^1({\bf (r+ff)}^2)\to 0.
$$
Now consider the short exact sequence
$$
0\to \frac{\bf f}{\bf r+ff}\otimes \frac{\bf ff}{\bf fr+fff}\to \frac{\bf f}{\bf r+ff}\otimes\frac{\bf ff+r}{\bf fr+fff}\to \frac{\bf f}{\bf r+ff}\otimes \frac{\bf ff+r}{\bf ff}\to 0
$$
The left hand term has zero limits, since it is isomorphic to $G_{ab}\otimes F_{ab}\otimes G_{ab}$. Since the diagonal action of $F$ on the middle and the right head terms are trivial, they are isomorphic to $\frac{\bf f(ff+r)}{\bf (r+ff)^2}$ and $\frac{\bf f(ff+r)}{\bf rr+fff}$ respectively. Hence,
\begin{align*}
& {\sf lim}^1({\bf (r+ff)}^2)={\sf lim}^1({\bf rr+fff})={\sf Tor}(G_{ab}, G_{ab}),\\
& {\sf lim}^2({\bf (r+ff)}^2)={\sf lim}^2({\bf rr+fff})=G_{ab}\otimes G_{ab}.
\end{align*}
We obtain the needed description
$$
{\sf lim}^2({\bf rr+ffr+rff})=G_{ab}\otimes G_{ab}
$$
and the short exact sequence
$$
0\to \frac{{\bf g}^2\otimes_{\mathbb Z[G]}{\bf g^2}}{\approx}\to {\sf lim}^1({\bf rr+ffr+rff})\to {\sf Tor}(G_{ab}, G_{ab})\to 0.
$$

\vspace{.5cm}\noindent
{\bf frr+rfr:} There is an isomorphism
$$
\frac{\bf ff}{\bf fr+rf}\otimes_{\mathbb Z[F]}{\bf r}=\frac{\bf ffr}{\bf frr+rfr}.
$$
This is a particular case of the functor $A\otimes_{\mathbb Z[F]}{\bf r}$, where $A$ is a constant, in this case $A=\g\otimes_{\Z G} \g$.
Since ${\sf lim}^{\bullet}\mathbb Z[F]=\mathbb Z$ is of finite global dimension and $\bf r$ is a free $\mathbb Z[F]$ -module, K\"unneth theorem 3.2
can be applied to $A\otimes_{\mathbb  Z[F]} \bf r$, and it degenerates to a series of usual K\"unneth short exact sequences
$$
\bigoplus_{i+j=n}{\sf lim}^i A\otimes {\sf lim}^j {\bf r} \hookrightarrow {\sf lim}^n A\otimes_{\mathbb Z[F]}{\bf r}\twoheadrightarrow \bigoplus_{i+j=n+1}{\sf Tor}
({\sf lim}^i A, {\sf lim}^j {\bf r})
$$
which computes the only non-trivial higher limit as
$$
\sf lim^1 A\otimes_{\mathbb Z[F]}{\bf r} = A\otimes {\bf g}.
$$
In this way, we obtain the description
\begin{align*}
& {\sf lim}^2({\bf frr+rfr})=({\bf g}\otimes_{\mathbb Z[G]}{\bf g})\otimes {\bf g}, &
& {\sf lim}^i({\bf frr+rfr})=0,\ i\neq 2.
\end{align*}
In the same way we have
\begin{align*}
& {\sf lim}^2({\bf rr+ffr})=G_{ab}\otimes {\bf g},  &
& {\sf lim}^i({\bf rr+ffr})=0,\ i\neq 2.
\end{align*}

\vspace{.5cm}\noindent
{\bf rff+frr:} There is an isomorphism
$$
\frac{\bf r}{\bf fr+rf}\otimes \frac{\bf r+ff}{\bf ff}=\frac{\bf rr+rff}{\bf rff+frr}.
$$
We have the following descriptions of the limits of above terms
\begin{align*}
& {\sf lim}\frac{\bf r}{\bf fr+rf}=H_2(G),&
& {\sf lim}^1\frac{\bf r}{\bf fr+rf}=G_{ab},&
& {\sf lim}^i\frac{\bf r}{\bf fr+rf}=0,\ i\geq 2
\end{align*}
and
\begin{align*}
& {\sf lim}^1 \frac{\bf r+ff}{\bf ff}=G_{ab},&
& {\sf lim}^i \frac{\bf r+ff}{\bf ff}=0,\ i\neq 1.
\end{align*}
As noted before, the abelian group $\frac{\bf r+ff}{\bf ff}={\bf \frac{r}{r\cap ff}}$ is torsion-free, hence the K\"unneth formula implies the following
\begin{align*}
& {\sf lim}\frac{\bf rr+rff}{\bf rff+frr}={\sf Tor}(H_2(G), G_{ab}),\\
& H_2(G)\otimes G_{ab} \hookrightarrow {\sf lim}^1\frac{\bf rr+rff}{\bf rff+frr}\twoheadrightarrow {\sf Tor}(G_{ab}, G_{ab}),\\
& {\sf lim}^2\frac{\bf rr+rff}{\bf rff+frr}=G_{ab}\otimes G_{ab},\\
& {\sf lim}^i\frac{\bf rr+rff}{\bf rff+frr}=0,\ i\geq 3.
\end{align*}
Comparing this description with the values of ${\sf lim}^i({\bf rr+rff})={\sf lim}^i({\bf rr+ffr}),$ we obtain the following:
\begin{align*}
& {\sf lim}^1({\bf rff+frr})={\sf Tor}(H_2(G), G_{ab}),\\
& F\hookrightarrow {\sf lim}^2({\bf rff+frr})\twoheadrightarrow {\sf ker}\{{\bf g}\otimes G_{ab}\twoheadrightarrow G_{ab}\otimes G_{ab}\},\\
& H_2(G)\otimes G_{ab}\hookrightarrow F\twoheadrightarrow {\sf Tor}(G_{ab}, G_{ab}).
\end{align*}

\vspace{.5cm}\noindent
{\bf  ffr+frf+rff+rr:} Consider the short exact sequence
$$
\frac{\bf fff}{\bf ffr+frf+rff+rr\cap fff}\hookrightarrow \frac{\bf ff}{\bf ffr+frf+rff+rr}\twoheadrightarrow \frac{\bf ff}{\bf rr+fff}.
$$
The left hand term is a natural quotient of ${\bf g}\otimes_{\mathbb Z[G]}{\bf g}\otimes_{\mathbb Z[G]}\bf g$, hence,
\begin{align*}
& {\sf lim}^2({\bf ffr+frf+rff+rr})={\sf lim}^2({\bf rr+fff})=G_{ab}\otimes G_{ab},\\
& {\sf lim}^2({\bf ffr+frf+rff+rr})=0,\ i\geq 3.
\end{align*}
Next observe that,
$$
{\bf rr\cap fff}\subset {\bf ffr+frf+rff}.
$$
This follows from the identification of the intersection of augmentation ideals\footnote{In the free group ring $\mathbb Z[F],$ ${\bf rr}=\Delta^2(R)+{\bf rrf}.$}
$\Delta^2(R)\cap {\bf fff}=\Delta^3(R)+\Delta(R)\Delta(R\cap [F,F])+\Delta([R,R]\cap [[F,F],F])$
and the identity $R'\cap \gamma_3(F)=[R\cap F', R]$\footnote{A simple proof of this identity is the following. Observe that, $\Lambda^2(R/(R\cap [F,F]))=\frac{[R,R]}{[R,R\cap [F,F]]},$ where $\Lambda^2$ is the exterior square, and $\Lambda^2(F_{ab})=[F,F]/[[F,F],F]$. Now the needed identity follows from the inclusion $\Lambda^2(R/(R\cap [F,F]))\hookrightarrow \Lambda^2(F_{ab}),$ which is induced by the inclusion $R/R\cap [F,F]\hookrightarrow F_{ab}$.}, see \cite{karan2002some}. Hence, the left hand term in the above short exact sequence is ${\bf g}\otimes_{\mathbb Z[G]}{\bf g}\otimes_{\mathbb Z[G]}\bf g$ itself
and we have a short exact sequence
$$
{\bf g}\otimes_{\mathbb Z[G]}{\bf g}\otimes_{\mathbb Z[G]}{\bf g}\hookrightarrow {\sf lim}^1({\bf ffr+frf+rff+rr})\twoheadrightarrow {\sf Tor}(G_{ab}, G_{ab}).
$$
We collect the results in the following table. By $F"\oplus"G$ we mean an extension of the form $F\hookrightarrow * \twoheadrightarrow G$. 

\newpage 

\vspace{.5cm}
\makebox[\linewidth]{
\begin{tabular}{|r||c|c|c|c|c|c|}
\hline  {\bf fr-code} & ${\bf lim}^1$&${\bf lim}^2$&${\bf
lim^3}$
\\ \hline {\bf r} & ${\bf g}$ & 0 & 0
\\ \hline {\bf rr} & 0 & ${\bf g}\otimes {\bf g}$ & 0
\\ \hline {\bf rrr} & 0 & 0 & ${\bf g}\otimes {\bf g}\otimes {\bf g}$
\\ \hline {\bf fr+rf} & ${\bf g}\otimes_{\mathbb Z[G]}{\bf g}$ & 0 & 0
\\ \hline {\bf ffr+frf+rff} & ${\bf g}\otimes_{\mathbb Z[G]}{\bf g}\otimes_{\mathbb Z[G]}{\bf g}$ & 0 & 0
\\ \hline {\bf r+ff} & $G_{ab}$ & 0 & 0
\\ \hline {\bf r+fff} & ${\bf g}/{\bf g}^3$ & 0 & 0
\\ \hline {\bf rf+ffr} & ${\bf g}^2\otimes_{G}{\bf g}$ & 0 & 0
\\ \hline {\bf fr+rf+fff} & $G_{ab}\otimes G_{ab}$ & 0 & 0
\\ \hline {\bf rr+fff} & ${\sf Tor}(G_{ab},G_{ab})$ & $G_{ab}\otimes G_{ab}$ & 0
\\ \hline {\bf rr+frf} & $H_3(G)$ & ${\bf g}\otimes_{G}{\bf g}$ & 0
\\ \hline {\bf rrf+frr} & $H_4(G)$ & $({\bf g}\otimes {\bf g}\otimes {\bf g})_G$ &
0
\\ \hline {\bf rfr+frf} & ${\sf coker}\{H_3(G)\otimes G_{ab}\to $ &
${\sf im}\{H_2(G)\otimes G_{ab}\to$  & 0
\\ & $H_2\left(G, {\bf g}\otimes_{\mathbb Z[G]}{\bf g}\right)\}$ & $H_1\left(G, {\bf g}\otimes_{\mathbb Z[G]}{\bf g}\right)\} "\oplus" {\bf g}^2\otimes_{\mathbb Z[G]}\bf g$ &
\\ \hline {\bf rff+ffr} & $\frac{{\bf g}^2\otimes_{\mathbb Z[G]}{\bf g}^2}{\sim}$ & 0 & 0
\\ \hline {\bf rr+frf+rff} & $H_2(G, G_{ab})$ & $G_{ab}\otimes G_{ab}$ & 0
\\ \hline {\bf rr+ffr} & 0 & $G_{ab}\otimes {\bf g}$& 0
\\ \hline {\bf rfr+frr} & 0 & $({\bf g}\otimes_{\mathbb Z[G]}{\bf g})\otimes {\bf g}$ & 0
\\ \hline {\bf rr+ffr+rff} & $\frac{{\bf g}^2\otimes_{\mathbb Z[G]}{\bf g}^2}{\approx}\ "\oplus"\ {\sf Tor}(G_{ab},G_{ab})$ & $G_{ab}\otimes G_{ab}$ &  0
\\ \hline {\bf rr+ffr+frf+rff} & ${\bf g}\otimes_{\mathbb Z[G]}{\bf g}\otimes_{\mathbb Z[G]}{\bf g}\ "\oplus"\ {\sf Tor}(G_{ab},G_{ab})$& $G_{ab}\otimes G_{ab}$ & 0
\\ \hline {\bf rff+frr} & ${\sf Tor}(H_2(G), G_{ab})$ & $H_2(G)\otimes G_{ab}\ "\oplus"\ {\sf Tor}(G_{ab}, G_{ab})$ & 0
\\ & & $"\oplus"\ {\sf ker}\{{\bf g}\otimes G_{ab}\twoheadrightarrow G_{ab}\otimes G_{ab}\}$ &
\\ \hline {\bf rrf+rfr+frr} & $H_2(G,H_2(G))\ "\oplus"\ {\sf Tor}(G_{ab}, H_2(G))$ & $H_2(G)\otimes G_{ab}\ "\oplus"\ H_2(G,G_{ab})$ & 0\\
& & $"\oplus"\ {\sf ker}\{{\bf g\otimes g}\to G_{ab}\otimes G_{ab}\}$ & \\
\hline
\end{tabular}
}

\printbibliography	
\end{document}